\documentclass[11pt]{amsart}
\usepackage{amssymb}
\usepackage{amsmath}
\usepackage{graphicx}
\usepackage{hyperref}
\usepackage{xcolor}
\usepackage{soul}
\usepackage{tikz}

\newtheorem{thm}{Theorem}[section]
\newtheorem{lem}[thm]{Lemma}
\newtheorem{fact}[thm]{Fact}
\newtheorem{cor}[thm]{Corollary}
\newtheorem{Q}[thm]{Question}
\newtheorem{Def}[thm]{Definition}
\newtheorem{prop}[thm]{Proposition}
\newtheorem{rem}[thm]{Remark}

\newcommand{\bdfn}{\begin{Def} \rm}
\newcommand{\edfn}{\end{Def}}

\newcommand{\beqa}{\begin{eqnarray*}}
\newcommand{\eeqa}{\end{eqnarray*}}

\newcounter{cnt1}
\newcounter{cnt2}
\newcounter{cnt3}
\newcounter{cnt4}
\newcommand{\blr}{\begin{list}{$($\roman{cnt1}$)$} {\usecounter{cnt1}
 \setlength{\topsep}{0pt} \setlength{\itemsep}{0pt}}}
\newcommand{\blR}{\begin{list}{\Roman{cnt4}.\ } {\usecounter{cnt4}
 \setlength{\topsep}{0pt} \setlength{\itemsep}{0pt}}}
\newcommand{\bla}{\begin{list}{$(\alph{cnt2})$} {\usecounter{cnt2}
 \setlength{\topsep}{0pt} \setlength{\itemsep}{0pt}}}
\newcommand{\bln}{\begin{list}{$($\arabic{cnt3}$)$} {\usecounter{cnt3}
 \setlength{\topsep}{0pt} \setlength{\itemsep}{0pt}}}
\newcommand{\el}{\end{list}}

\begin{document}

\title[\tiny{An extension of Phelps theorem}]{An extension of Phelps theorem to spaces of vector-valued functions}

\author[Dwivedi]{Saurabh Dwivedi}

\address{Shiv Nadar Institution of Eminence. Gautam Buddha Nagar, Delhi NCR, Uttar Pradesh-201314, India}

\email{sd605@snu.edu.in, saurabhdewedi876@gmail.com}

\subjclass[2020]{Primary 46A22, 46B10, 46B25; Secondary 46B20, 46B22. \hfill
\textbf{}}

\begin{abstract}
        In this paper, our main aim is to extend the classical theorem of Phelps (see \cite[pg.~50, Theorem~7]{HE}) on norm-attaining functionals from the space of scalar-valued continuous functions $C(\Omega)$ to its vector-valued counterpart $C(\Omega, X)$. In one of our main results, we give a complete characterization of norm-attaining functionals on $C(\Omega, X)$, assuming that $X^*$ has the Radon-Nikodym Property (RNP). For a general Banach space $X$, we investigate norm attainment for the points of weak$^*$ to weak continuity for the identity map $Id: (C(\Omega, X)_{1}^*, w^*) \rightarrow (C(\Omega, X)_{1}^*, w)$.
 
\end{abstract}
\maketitle
\section{Introduction}

Let $X$ be a real or complex Banach space, and $X^{*}$ denote the dual space of $X$. We denote by $S(X)$ and $X_1$ the unit sphere and the unit ball of $X$, respectively. Suppose $x^*\in X_{1}^*$. We say that the functional $x^*$ attains its norm on $S(X)$ if there exists a unit vector $x\in X$ such that $x^*(x)=1$. For an infinite compact Hausdorff space $\Omega$, let $C(\Omega)$ denote the space of all complex-valued continuous functions equipped with the supremum norm. For a nonzero point $x^* \in C(\Omega)^*$, we know that there exists a positive regular Borel measure $\mu$ on $\Omega$ such that $x^*(f) = \int f h \, d\mu$ for all $f \in C(\Omega)$ and $\|x^*\| = \|\mu\|$, where $h$ is a Borel measurable function on $\Omega$ with $|h| = 1$, $|\mu|$-almost everywhere. In one of the classic results of Phelps (see \cite[pg.~50, Theorem~7]{HE}), he has shown that $x^*$ attains its norm if and only if there exists $g \in C(\Omega)$ with $|g| = 1$, $|\mu|$-almost everywhere, $\|g\| \leq 1$, and $x^*(f) = \int f g \, d\mu$ for all $f \in C(\Omega)$. This result is significant in view of the Bishop-Phelps theorem, that is, for any Banach space $X$, the set of norm-attaining functionals on $X$ forms a norm-dense subset of $X^*$. For a Banach space $X$, let $C(\Omega, X)$ denote the space of $X$-valued continuous functions on $\Omega$, equipped with the supremum norm. The dual of $C(\Omega, X)$ can be identified with $M(\Omega, X^*)$, the space of $X^*$-valued regular Borel measures, equipped with the total variation norm (see \cite[Theorem~1.7.1]{CP}), when $X=\mathbb{R}$ or $\mathbb{C}$, we write $M(\Omega)$. For a detailed study of the theory of vector-valued measures, we refer the reader to \cite{DU}.

If \( \mu \in M(\Omega) \) and \( x^* \in X^* \), we define an associated functional by
\[
(\mu \otimes x^*)(f) = \int_{} x^*(f(\omega)) \, d\mu(\omega), \quad \text{for all } f \in C(\Omega, X).
\]
As a specific instance, when \( \omega \in \Omega \) and \( x^* \in X^* \), the functional defined by \( (\delta_\omega \otimes x^*)(f) = x^*(f(\omega)) \) corresponds to the evaluation at the point \( \omega \), where \( \delta_\omega \) denotes the Dirac measure at \( \omega \).

Similarly, given any \( x^{**} \in X^{**} \) and Borel subset \( A \subseteq \Omega \), we define a functional on \( M(\Omega, X^*) \) by setting
\[
(x^{**} \otimes \chi_A)(F) = x^{**}(F(A)), \quad \text{for all } F \in M(\Omega, X^*)
\]
and it is straightforward to verify that this indeed defines an element of the bidual \( C(\Omega, X)^{**} \). One of our main goals in this paper is to establish a vector-valued version of Phelps' theorem. Specifically, we show that for a Banach space $X$ with $X^*$ having the Radon-Nikodym Property (RNP). For any nonzero bounded linear functional $F$ on $C(\Omega, X)$, if $|F|$ denotes the total variation of $F$ and $h: \Omega \to X^*$ is a Borel measurable function with $\|h(\omega)\| = 1$ $|F|$-almost everywhere and $F = h d|F|$, then $F$ attains its norm on $S(C(\Omega, X))$ if and only if there exists $g \in S(C(\Omega, X))$ such that $h(\omega)(g(\omega)) = 1$ $|F|$-almost everywhere. Then we investigate for which functionals $F$ the norm is attained.

Recall from \cite{SMR} that a point $x^* \in X_{1}^*$ is called a \textit{w}$^*$-\textit{w} point of continuity (PC) if the identity map $Id: (X_{1}^{*}, \textit{w}^*) \rightarrow (X_{1}^{*}, \textit{w})$ is continuous at $x^*$, and a \textit{w}$^*$-PC if the identity map $Id: (X_{1}^{*}, \textit{w}^*) \rightarrow (X_{1}^{*}, \|\cdot\|)$ is continuous at $x^*$. We investigate norm attainment of elements in $C(\Omega, X)^*$ in the context of \textit{w}$^*$-\textit{w} PCs. Motivated by the work of Hu and Smith (see \cite{HS94}), who studied the points of continuity of the identity map $Id: (X_{1}^{*}, \textit{w}^*) \rightarrow (X_{1}^{*}, \|\cdot\|)$. The authors of \cite{SMR} analyze the structure of \textit{w}$^*$-\textit{w} points of continuity of the identity map $Id: (X_{1}^{*}, \textit{w}^*) \rightarrow (X_{1}^{*}, \textit{w})$ for non-reflexive Banach spaces. Subsequently, we use these structural theorems to investigate the norm attainment of such points of continuity. We recall that a compact Hausdorff space $\Omega$ is extremally disconnected if the closure of every open subset of $\Omega$ is again open in $\Omega$. One of the most standard examples of such a space is the Stone-\v{C}ech Compactification of a discrete set $S$, denoted by $\beta S$. More generally, if $\mu$ is a positive measure on $\Omega$, then the Stone space of the Banach algebra $L^{\infty}(\mu)$ is an extremally disconnected space (see \cite[pg.~119]{HE}). For a detailed discussion on the structure and significance of extremally disconnected spaces in functional analysis, we refer the reader to \cite[pg.~40]{HE}. Throughout, we use $I$ to denote the set of all isolated points in $\Omega$. For any subset $A\subset \Omega$, we denote by $A^c$, the set $\{\omega\in \Omega: \omega\notin A\}$ and $\overline{A}$ denotes the closure of $A$ in $\Omega$. Note that, if $\Omega$ is an extremally disconnected compact Hausdorff space, then $\overline{I}$ is a clopen subset of $\Omega$.

We denote by WC$(\Omega, X)$, the space of $X$-valued functions on $\Omega$ that are continuous when $X$ has the weak topology, equipped with the supremum norm. We also consider the vector-valued version of Phelps' theorem in this case. In \cite[pg.~212]{HS94}, authors have shown that an element $\Lambda\in C(\Omega, X)_{1}^*$ is \textit{w}$^*$-PC if and only if it has the form $\sum_{k\in I}\delta_{k}\otimes x_{k}^*$, where $I=\{k\in \Omega:k \textnormal{ is an isolated point of }\Omega\}$ and for each $k\in I$, either $x_{k}^*=0$ or $\frac{x_{k}^*}{\|x_{k}^*\|}$ is a \textit{w}$^*$-PC in $X_{1}^*$ and $\sum_{k\in I}\|x_{k}^*\|=1$. In \cite{R2}, Rao has extended the above result to classify \textit{w}$^*$-PC's of the dual unit ball WC$(\Omega,X)^*_{1}$. We provide a complete classification of \textit{w}$^*$-\textit{w} PCs in the dual unit balls $C(\Omega, X)_{1}^*$ and $\textnormal{WC}(\Omega, X)_{1}^*$, when $\Omega$ is extremally disconnected. 

We denote the canonical embedding of $X$ into its bidual $X^{**}$ by identifying $X$ with its image under the map $J_{X}:X\to X^{**}$, where $J_{X}$ is the canonical embedding. For clarity, we sometimes write $J_{X}(X)$ to explicitly refer to the embedded copy of $X$ in $X^{**}$. As shown in \cite[Lemma 2.14]{HWW}, a functional $x^*\in S(X^*)$ is \textit{w}$^*$-\textit{w} PC  if and only if it admits a unique norm preserving extension from $X$ to its bidual $X^{**}$. In this case, the canonical embedding $J_{X^*}: X^*\rightarrow X^{***}$ provides precisely this unique extension, meaning that $J_{X^*}(x^*)$ is the only norm preserving extension of $x^*$ from $X$ to $X^{**}$.

Let $\mathcal{K}(X, Y)$ and $\mathcal{F}(X, Y)$ denote the spaces of compact and weakly compact operators from $X$ to $Y$, respectively. Taking $Y=C(\Omega)$, it is well known that the spaces mentioned above can be identified with the duals $C(\Omega, X^*)$ and WC$(\Omega, X^*)$,  respectively (see \cite[page 490]{DD}). This identification serves as a key motivation for our study of \textit{w}$^*$-\textit{w} PCs in the dual unit balls $C(\Omega, X)^*_{1}$ and WC$(\Omega, X)^*_{1}$. For a discrete set $S$, we define the space $c_0({S,X})$ as the set $\{(x_{s})_{s\in S}:x_{s}\in X \text{ and }\|x_{s}\|_{s\in S}\rightarrow 0\}$, endowed with the supremum norm. Similarly, we define the spaces $\ell^{1}(S,X)=\{(x_{s})_{s\in S}:x_{s}\in X, \sum_{s\in S}\|x_{s}\|<\infty\}$, endowed with the $\ell^{1}$ and $\ell^{\infty}(S,X)=\{(x_{s})_{s\in S}:x_{s}\in X, \text{ sup}_{s\in S}\|x_{s}\|<\infty\}$, endowed with the $\ell^{\infty}$ norm. For $s \in S$ and $x \in X$, we denote by $e_s x$ the element of $\ell^1(S, X)$ which takes the value $x$ at $s$ and $0$ elsewhere. We denote by $c$ the space of all convergent sequences in the complex plane, endowed with the supremum norm.

In Section~2, we first establish how vector-valued regular Borel measures of bounded variation act on $C(\Omega, X)$, that is, we describe the action of $M(\Omega, X^*)$ on $C(\Omega, X)$, under the assumption that $X^*$ has the RNP. In the same section, we present the classification of norm-attaining functionals on $C(\Omega, X)^*$. We refer to this as a vector-valued version of Phelps' theorem (see Theorem~\ref{T34}).

In section~3, we analyze the structure of \textit{w}$^*$-\textit{w} PCs in the dual unit ball $C(\Omega, X)^*_{1}$, when $\Omega$ is extremally disconnected. Finally, we use the structure of \textit{w}$^*$-\textit{w} PCs in the dual unit ball $C(\Omega, X)_{1}^*$ to derive explicit conditions for norm attainment. In Theorem~\ref{T0}, we provide a complete characterization of \textit{w}$^*$-\textit{w} PCs of the dual unit ball $C(\Omega, X)_{1}^*$ that attain their norm on $S(C(\Omega, X))$, under some condition on $\Omega$.

In section~4, we analyze the structure of \textit{w}$^*$-\textit{w} PCs in the dual unit ball WC$(\Omega,X)^*_{1}$, when $\Omega$ is extremally disconnected. Also, we provide the results similar to the one established in section~3 for norm attainment in the space WC$(\Omega, X)^*_{1}$.

\section{Phelps' Theorem for the space $C(\Omega, X)$}
A linear functional $x^*\in X^*$ is said to be a support functional of $X_{1}$ if $x^*\neq 0$, and it attains its norm on $X_{1}$. We begin this section by recalling two well-known theorems from \cite[pg.~20]{Ho} and \cite[pg.~50]{HE} in our setup.
\begin{thm}[Bishop-Phelps]\label{T31}
    Let $A$ be a closed convex set in the real Banach space and let $\phi\in S(X^*)$ be bounded above on $A$. Then for any $\delta\in (0,1)$ there exists a support functional $\varphi$ of $A$ with $\|\phi-\varphi\|<\delta$.
\end{thm}
\begin{thm}[Phelps]\label{T32}
    Let $\Omega$ be a compact Hausdorff space, $x^*$ a non zero bounded linear functional on $C(\Omega)$, $\mu$ a positive regular Borel measure on $\Omega$ and $h$ a Borel measurable function on $\Omega$ with $|h|=1$ and $x^*(f)=\int fhd\mu$ for all $f\in C(\Omega)$ and $\|x^*\|=\|\mu\|$. Then $x^*$ attains its norm if and only if there is a $g\in C(\Omega)$ with $|g|=1$, $|\mu|$-almost everywhere,  $\|g\|\leq 1$ and $x^*(f)=\int fgd\mu$ for all $f\in C(\Omega)$.
\end{thm}
In Theorem~\ref{T32}, Phelps characterizes the norm-attaining functionals on the space $C(\Omega)$. However, extending this characterization to the vector-valued setting $C(\Omega, X)$ is significantly more challenging, as the dual space $C(\Omega, X)^*$ is more complex than its scalar-valued version  $C(\Omega)$. As discussed in the introduction, the dual $C(\Omega, X)^*$ can be identified with the space $M(\Omega, X^*)$ of regular Borel $X^*$-valued measures. Let $F\in C(\Omega, X)^*$ be such a measure and let $x\in X$. We can define a scalar-valued measure $x\otimes F:\Omega\to \mathbb{C}$ by $x\otimes F(A)=F(A)(x)$ for all Borel sets $A\subseteq \Omega$. It is easy to see that $x\otimes F\in M(\Omega)$. We denote by $|F|$ the total variation of $F$. Similarly, given a $f\in C(\Omega)$ and $x\in X$, we can define $x\otimes f:\Omega\to X$ given by $\omega\to f(\omega)x$. Clearly, $x\otimes f\in C(\Omega, X)$. In the following lemma, we make use of such objects to understand how elements of $C(\Omega, X)^*$ act on $C(\Omega, X)$, when $X^*$ has the RNP. 
\begin{lem}\label{L33}
    Let $X$ be such that $X^*$ has the RNP and let $F\in C(\Omega, X)^*$. If $h:\Omega\to X^*$ is the Radon-Nykodym derivative of $F$ (i.e, $F=hd|F|$), then $F(f)=\int h(\omega)(f(\omega))d|F|$ for all $f\in C(\Omega, X)$.
\end{lem}
\begin{proof}
    Let $S=\{x\otimes f: x\in X\text{ and }f\in C(\Omega)\}\subseteq C(\Omega, X)$. Then we have $\overline{span}(S)=C(\Omega, X)$ (see \cite[pg. 49-50]{RR}). For $x\otimes f \in S$, we have $F(x\otimes f)=\int x\otimes fdF=\int f(\omega)d(x\otimes F)$ (see \cite{WH}). By the hypothesis, $h$ is Bochner integrable; in particular, it is Pettis integrable and $\|h(\omega)\|=1$, $|F|$-almost everywhere. From \cite[proposition~3.8]{RR}, we have $F(E)(x)=(\int_{E} h(\omega)d|F|)(x)=\int_{E} h(\omega)(x)d|F|$ for each measurable subset $E\subseteq \Omega$. So we get $d(x\otimes F)=h(\omega)(x)d|F|$. Thus,
    $$
        F(x\otimes f)=\int f(\omega)d(x\otimes F)
        =\int f(\omega)h(\omega)(x)d|F|
        =\int h(\omega)((x\otimes f)(\omega))d|F|.
    $$ 
    Since $span(S)$ is dense in the space $C(\Omega, X)$ and $F$ is continuous, we have $F(f)=\int h(\omega)(f(\omega))d|F|$ for all $f\in C(\Omega, X)$. This completes the proof.
\end{proof}

Now we present the main theorem of this section, which we refer to as a vector-valued analogue of Phelps's Theorem. In this theorem, we use Lemma~\ref{L33} to characterize norm-attaining functionals on $C(\Omega, X)$, when $X^*$ has the RNP.
\begin{thm}\label{T34}
    Let $X$ be a Banach space such that $X^*$ has the RNP, $F$ a non zero bounded linear functional on $C(\Omega, X)$, $|F|$ be the total variation of $F$ and $h:\Omega\to X^*$ a Borel measurable function with $\|h(\omega)\|=1$, $|F|$-almost everywhere and $F=hd|F|$. Then $F$ attains its norm on $S(C(\Omega, X))$ if and only if there exists a $g\in S(C(\Omega, X))$ such that $h(\omega)(g(\omega))=1$, $|F|$-almost everywhere.
\end{thm}
\begin{proof}
    Let $F$ attains its norm at $g\in S(C(\Omega, X))$. Since $\|h(\omega)\|=1$, $|F|$-almost everywhere and $\|g\|=1$, we get that $|h(\omega)(g(\omega))|\leq 1$, $|F|$-almost everywhere. From Lemma~\ref{L33}, we have

    \begin{align*}
        \|F\|=F(g)&=\int h(\omega)(g(\omega))d|F|\\
        &\leq\int |h(\omega)(g(\omega))|d|F|\\
        &=\int 1d|F|\\
        &=\|F\|\\
    \end{align*}

    This implies $\int (1-|h(\omega)(g(\omega)))|d|F|=0$. It gives that $|h(\omega)(g(\omega))|=1$, $|F|$-almost everywhere. By using a similar argument, we get that $Re(h(\omega)(g(\omega)))=1$. Combining last two equalities, we get that $h(\omega)(g(\omega))=1$, $|F|$-almost everywhere. Conversely, if there exists a function $g\in S((\Omega,X))$ such that $h(\omega)(g(\omega))=1$, $|F|$-almost everywhere, then it is easy to see that $F$ attains its norm at $g$. This completes the proof.
\end{proof}

\section{Phelps' Theorem for norm-attaining functionals which are \textit{w}$^*$-\textit{w} PC}
The main goal of this is to find conditions under which one can explicitly find a continuous function $f$, where a norm-attaining functional $F$ attains its norm at $f$ without assuming the global condition, such as RNP. We begin by recalling the following set of results from \cite{HWW} and \cite{R3}, which will be utilised in subsequent arguments. For any subset $E\subseteq X^*$, we denote by ext$(E)$ the set of all extreme points of $E$ and by co$(E)$ the convex hull of $E$.  Let $Y\subseteq X$ be a closed subspace. We say that $Y$ is an M-ideal if there exists a norm one projection $P:X^*\to X^*$ such that ker($P$)=$Y^{\perp}$ and $X^*=Y^{\perp}\oplus_{1}Y^*$.
\begin{lem}\cite[Theorem~6.2]{SMR}\label{L21}
    For an M-ideal $Y\subseteq X$. If $x^*\in S(Y^*)$ is a \textit{w}$^*$-\textit{w} PC, then $x^*\in S(X^*)$ is a \textit{w}$^*$-\textit{w} PC.
\end{lem}
In the next lemma, we identify $X$ with its canonical image in the bidual $X^{**}$.
\begin{lem}\cite[Lemma 2.14]{HWW}\label{C22}
    Let $x^*\in S(X^*)$. Then $x^*$ is a \textit{w}$^*$-\textit{w} PC if and only if $x^*$ has a unique norm preserving extension from $X$ to $X^{**}$.
\end{lem}
\begin{fact}\cite[proposition~1.17]{HWW}
Let $Y\subseteq X$ be an M-ideal in $X$ and let $W\subseteq X$ be a closed subspace such that $Y\subseteq W\subseteq X$. Then $Y$ is an M-ideal in $W$.
\end{fact}
In the following theorem, we use Lemma~\ref{L21} and Lemma~\ref{C22} to provide an explicit form of \textit{w}$^*$-\textit{w} PCs in the dual unit ball $C(\Omega, X)_{1}^*$. We recall that the set of extreme points of $C(\Omega, X)_{1}^*$ is given by $\text{ext}(C(\Omega, X)_{1}^*)=\{\delta_{\omega}\otimes x^*:\omega\in \Omega, x^*\in \text{ext}(X_{1}^*)\}$ as established in \cite[Theorem~1.1(a)]{RS}.
\begin{thm}\label{T23}
    Let $\Lambda\in C(\Omega, X)_{1}^*$ be a \textit{w}$^*$-\textit{w} PC. Then $\Lambda=\sum_{i=1}^{\infty}\alpha_{i}\delta_{\omega_{i}}\otimes x_{{i}}^*$, where $\alpha_{i}\in(0,1]$ such that $\sum_{i=1}^{\infty}\alpha_{i}=1$, $x_{{i}}^*\in X_{1}^*$ is \textit{w}$^*$-\textit{w} PC for each $i\in \mathbb{N}$ and each $\omega_{i}\in \Omega$ is an isolated point.
\end{thm}
\begin{proof}
    Let $\Lambda$ be a \textit{w}$^*$-\textit{w} PC in $C(\Omega, X)^*_{1}$. Since $C(\Omega, X)^*_{1}$ is weak$^*$ compact, by Krein-Milman theorem (see \cite[Theorem~3.23]{RW}), we have  $\Lambda\in \overline{co}^{w^*}\{\delta_{\omega}\otimes x^*:x^*\in \text{ext}(X_{1}^*)\text{ and } \omega\in \Omega\}$, and hence, $\Lambda\in \overline{co}^\textit{w}\{\delta_{\omega}\otimes x^*:x^*\in \text{ext}(X_{1}^*)\text{ and } \omega\in \Omega\}=\overline{co}^{\|.\|}\{\delta_{\omega}\otimes x^*:x^*\in \text{ext}(X_{1}^*)\text{ and } \omega\in \Omega\}$. This implies that supp$(\Lambda)$ is at most countable. Let supp$(\Lambda)=\{\omega_{n}:n\in \mathbb{N}\}$. Since $\Lambda$ is a \textit{w}$^*$-\textit{w} PC, we have $\|\Lambda\|=1$ (see \cite[Lemma 2.4]{SMR}) i.e, $\sum_{i=1}^{\infty}\|\Lambda(\{\omega_{i}\})\|=1$. Let $x_{n}^*=\frac{\Lambda(\{\omega_{n}\})}{\|\Lambda(\{\omega_{n}\})\|} \text{ and } \alpha_{n}=\|\Lambda(\{\omega_{n}\})\|$. Then it is easy to see that $\Lambda=\sum_{i=1}^{\infty}\alpha_{i}\delta_{\omega_{i}}\otimes x_{i}^*$. Now we will show that each $\omega_{n}\in \Omega$ is an isolated point and each $x_{n}^*$ is a \textit{w}$^*$-\textit{w} PC in $X_{1}^*$. Fix $i_{0}\in \mathbb{N}$ and let $(\omega_{\alpha})_{\alpha\in J}\subseteq \Omega$ be a net of distinct points such that $\omega_{\alpha}\to \omega_{i_{0}}$. This implies $\delta_{\omega_{\alpha}}\to \delta_{\omega_{i_{0}}}$ in the weak$^*$ topology. For $\alpha\in J$, we consider $\Lambda_{\alpha}\in C(\Omega, X)_{1}^*$ given as $\Lambda_{\alpha}=\left(\sum_{j\neq i_{0}}\alpha_{j}\delta_{\omega_{j}}\otimes x_{j}^*\right)+\alpha_{i_{0}}\delta_{\omega_{\alpha}}\otimes x_{i_{0}}^*$. It is easy to see that $\|\Lambda_{\alpha}\|=1$ and $\Lambda_{\alpha}\to \Lambda$ in the weak$^*$ topology. Thus, by the hypothesis, we get $\Lambda_{\alpha}\to \Lambda$ in the weak topology. Since $\|x_{i_{0}}^*\|=1$, we can choose $x_{i_{0}}\in X$ such that $x_{i_{0}}^*(x_{i_{0}})\neq0$. If we consider an element of the form $x_{i_{0}}\otimes\chi_{\omega_{i_{0}}}\in C(\Omega, X)^{**}$, then we have $x_{i_{0}}\otimes\chi_{\omega_{i_{0}}}(\Lambda_{\alpha})\to x_{i_{0}}\otimes\chi_{\omega_{i_{0}}}(\Lambda)$. Consequently, we get $x_{i_{0}}^*(x_{i_{0}})=0$, which is a contradiction to our choice of $x_{i_{0}}$. Thus $\omega_{i_{0}}$ is an isolated point. Similar argument will show that $x_{i_{0}}^*$ is a \textit{w}$^*$-\textit{w} PC in $X_{1}^*$. 
\end{proof}

The corollary below is an immediate consequence of Theorem~\ref{T23}.
\begin{cor}
    Let $X$ be a reflexive Banach space and let $\Lambda\in C(\Omega, X)_{1}^*$ is a \textit{w}$^*$-\textit{w} PC. Then $\Lambda=\sum_{i=1}^{\infty}\alpha_{i}\delta_{\omega_{i}}\otimes x_{{i}}^*$, where $\alpha_{i}\in(0,1]$ such that $\sum_{i=1}^{\infty}\alpha_{i}=1$, $x_{{i}}^*\in S(X^*)$ and each $\omega_{i}\in \Omega$ is an isolated point.
\end{cor}

For a general compact Hausdorff space $\Omega$ and a Banach space $X$, we do not know if the converse of Theorem~\ref{T23} is always true. In the following theorem, we show that if $\Omega$ is an extremally disconnected space, one can obtain the converse for any Banach space $X$.
\begin{thm}\label{T24}
    Let $\Omega$ be an extremally disconnected space and let $\Lambda \in C(\Omega, X)_{1}^*$ be such that $\Lambda=\sum_{i=1}^{\infty}\alpha_{i}\delta_{\omega_{i}}\otimes x_{{i}}^*$, where $\alpha_{i}\in(0,1]$ such that $\sum_{i=1}^{\infty}\alpha_{i}=1$, $x_{{i}}^*\in X_{1}^*$ is \textit{w}$^*$-\textit{w} PC for each $i\in \mathbb{N}$ and each $\omega_{i}\in \Omega$ is an isolated point. Then $\Lambda$ is a \textit{w}$^*$-\textit{w} PC.
\end{thm}
\begin{proof}
    
    First we assume that $I$ is dense in $\Omega$ and let $\Lambda=\sum_{i=1}^{\infty}\alpha_{i}\delta_{\omega_{i}}\otimes x_{{i}}^*$, where $\alpha_{i},\omega_{i} \text{ and } x_{i}^*$ satisfies the hypothesis of the theorem. Observe that $c_{0}(I, X)\subseteq C(\Omega, X)\subseteq l^{\infty}(I, X)$. Indeed, let $f\in c_{0}(I, X)$. Define $f':\Omega\to X$ by

    $$
       f'(\omega)=
    \begin{cases}
       f(\omega), & \omega \in I,\\
       0, & \omega\in \Omega\setminus I.
    \end{cases}
    $$
    It is enough to show that $f'$ is continuous on $\Omega\setminus I$. Let $\omega\in\Omega\setminus I$ and $(\omega_{\alpha})_{\alpha\in J}\subseteq \Omega$ be a net of distinct points such that $\omega_{\alpha}\to \omega$. Since $f$ vanishes at infinity and $I$ is a discrete set, the set $\{\omega_{\alpha}:\|f'(\omega_{\alpha})\|\geq\epsilon\}\subseteq I$ is finite for every $\epsilon>0$. This implies that $f'(\omega_{\alpha})\to 0=f'(\omega)$. Hence $f'\in C(\Omega, X)$ and $\|f\|=\|f'\|$. To establish the latter inclusion, one can use the restriction map. Since $c_{0}(I, X)\subseteq l^{\infty}(I, X)$ is an M-ideal, we have $c_{0}(I, X)\subseteq C(\Omega, X)$ is an M-ideal. It is enough to show that $\Lambda\in c_{0}(I, X)^*$ is \textit{w}$^*$-\textit{w} PC.

    Let $\hat{\Lambda}$ denotes the canonical image of $\Lambda$ in the space $c_{0}(I, X)^{***}$. We claim that $\hat{\Lambda}\in c_{0}(I, X)^{***}$ is the only norm preserving extension of $\Lambda$ from $c_{0}(I, X)$ to $c_{0}(I, X)^{**}$. If possible $\Lambda'\in c_{0}(I, X)^{***}$ is such that $\Lambda'|_{c_{0}(I, X)}=\Lambda$ and $\|\Lambda'\|=\|\Lambda\|=1$. We write $c_{0}(I, X)^{***}=l^{\infty}(I, X^{**})^*=c_{0}(I, X^{**})^{\perp}\oplus_{1}l^{1}(I, X^{***})$. Then $\Lambda'=\phi_{1}+\phi_{2}$, $\|\Lambda'\|=\|\phi_{1}\|+\|\phi_{2}\|=1$, where $\phi_{1}\in c_{0}(I, X^{**})^{\perp}$ and $\phi_{2}\in l^{1}(I, X^{***})$. Fix $i\in \mathbb{N}$ and consider $\phi_{2}(\omega_{i})\in X^{***}$ (${\omega_{i}}^{th}$ coordinate of $\phi_{2}$). Then for any $x\in X$, we have  $\phi_{2}(\omega_{i})(x)=\phi_{2}(e_{\omega_{i}}x)=\Lambda'(e_{\omega_{i}}x)=\Lambda(e_{\omega_{i}}x)=\alpha_{i}x_{i}^*(x)$. Thus we get $\phi_{2}(\omega_{i})|_{X}=\alpha_{i}x_{i}^*$, and hence, $\|\phi_{2}(\omega_{i})\|\geq\alpha_{i}$ for each $i\in \mathbb{N}$. Since 
    $$1\geq\|\phi_{2}\|\geq\sum_{i=1}^{\infty}\|\phi_{2}(\omega_{i})\|\geq\sum_{i=1}^{\infty}\alpha_{i}=1,$$ 
    
    we get that $\|\phi_{2}\|=1$. This implies $\phi_{1}=0 \text{ and } \Lambda'=\phi_{2}$. It is easy to see that supp$(\phi_{2})=\{\omega_{i}:i\in \mathbb{N}\}$. We have 
    $$\|\phi_{2}\|=1=\sum_{i=1}^{\infty}\alpha_{i}=\sum_{i=1}^{\infty}\|\phi_{2}(\omega_{i})\|.$$
    
    It gives that $\|\phi_{2}(\omega_{i})\|=\alpha_{i}$. By the hypothesis, each $x_{i}^*$ is \textit{w}$^*$-\textit{w} PC. By applying Lemma~\ref{C22} on $X^*$, we get $\phi_{2}(\omega_{i})=\alpha_{i}x_{i}^*$. This implies that $\phi_{2}=\Lambda'=\hat\Lambda$.  Hence $\Lambda$ has a unique norm preserving extension from $c_{0}(I, X)$ to $c_{0}(I, X)^{**}$. By applying Lemma~\ref{C22} on $c_{0}(I, X)^*$, we have $\Lambda\in c_{0}(I, X)^{*}_{1}$ is a \textit{w}$^*$-\textit{w} PC . Hence, the conclusion follows from Lemma~\ref{L21}.

    If $I$ is not dense in $\Omega$, since $\Omega$ is an extremally disconnected space, it follows that $\overline{I}$ is clopen. So we can write $C(\Omega, X)=C(\overline{I},X)\oplus_{\infty}C(\overline{I}^c,X)$. Now we may assume $\Lambda\in C(\overline{I},X)_{1}^*$. By applying part~(1) on the space $C(\overline{I},X)$ gives that $\Lambda\in C(\overline{I},X)_{1}^*$ is a \textit{w}$^*$-\textit{w} PC. Since an M-summand is also an M-ideal, the conclusion follows from Lemma~\ref{L21}.
\end{proof}

We recall from the introduction that $c_{0}$ denotes the space of all null sequences, equipped with the supremum norm. It is well-known that weak, weak$^*$ and norm topologies coincide on $S(c_{0}^*)$. For a Banach space $X$, we recall that $\mathcal{K}(X)$ denotes the space of all the compact operators on $X$, while $\mathcal{L}(X)$ denotes the set of all bounded operators on $X$. It is known (see \cite[pg.~291, Proposition~4.6]{HWW}) that if $\mathcal{K}(X)$ is an M-ideal in $\mathcal{L}(X)$, then weak$^*$ and norm topologies coincide on $S(X^*)$, and hence, weak and norm topologies coincide on $S(X^*)$. Let $X$ be a Banach space such that weak and norm topologies coincide on $S(X^*)$. The following proposition provides a situation where the converse of Theorem~\ref{T23} holds for each compact Hausdorff space $\Omega$.

\begin{prop}
   Let $X$ be such that the weak and norm topologies coincide on $S(X^*)$ and let $\Lambda\in C(\Omega, X)_{1}^*$. If $\Lambda=\sum_{i=1}^{\infty}\alpha_{i}\delta_{\omega_{i}}\otimes x_{{i}}^*$, where $\alpha_{i}\in(0,1]$ such that $\sum_{i=1}^{\infty}\alpha_{i}=1$, $x_{{i}}^*\in X_{1}^*$ is \textit{w}$^*$-\textit{w} PC for each $i\in \mathbb{N}$ and each $\omega_{i}\in \Omega$ is an isolated point, then $\Lambda$ is a \textit{w}$^*$-\textit{w} PC.
\end{prop}
\begin{proof}
    Let $\Lambda = \sum_{i=1}^{\infty} \alpha_{i} \delta_{\omega_{i}} \otimes x_{i}^*$, where $\alpha_{i} \in (0,1]$ with $\sum_{i=1}^{\infty} \alpha_{i} = 1$, $x_{i}^* \in X_{1}^*$ is a \textit{w}$^*$-\textit{w} PC for each $i \in \mathbb{N}$, and each $\omega_{i} \in \Omega$ is an isolated point. Since weak and norm topologies coincide on $S(X^*)$, we get that each $x_{i}^*\in X_{1}^*$ is a \textit{w}$^*$-PC. By \cite[Theorem~6]{HS94}, we have that $\Lambda\in C(\Omega, X)_{1}^*$ is a \textit{w}$^*$-PC, and hence, a \textit{w}$^*$-\textit{w} PC. This completes the proof.
\end{proof}

As stated in Theorem~\ref{T31}, the Bishop-Phelps theorem asserts that norm-attaining functionals in the unit sphere $S(X)$ are norm dense in $X^*$. However, in general, a functional $x^*\in X_{1}^*$ that attains its norm need not be a \textit{w}$^*$-\textit{w} PC. In the following note, we focus specifically on the intersection of these two notions, namely, functionals that are both norm-attaining and \textit{w}$^*$-\textit{w} PC. 

In Theorem~\ref{T23}, we established a stability result for \textit{w}$^*$-\textit{w} PCs in the space $C(\Omega, X)$. Using this result, we now derive a necessary condition under which a \textit{w}$^*$-\textit{w} PC in $S(C(\Omega, X)^*)$ attains its norm. Moreover, in this case, we establish an explicit connection between the norm attainment of a functional $\Phi \in C(\Omega, X)^*$ and the norm attainment of elements of $X^*$. This stands in contrast to Theorem~\ref{T34}. For an element $\Phi\in M(\Omega, X^*)$, we define, range$(\Phi)=\{\Phi(B)\in X^*: B \textnormal{ is a Borel subset of } \Omega \}$.
\begin{prop}\label{T35}
    Let $\Phi\in C(\Omega, X)_{1}^*$ be a \textit{w}$^*$-\textit{w} PC. If $\Phi$ attains its norm at  $f\in C(\Omega, X)_{1}$, then there exists a countable set $A=\{x_{i}:i\in \mathbb{N}\}_{}\subseteq X_{1}$ such that each element in range$(\Phi)$ attains its norm on $A$ and $\overline{A}\subseteq X_{1}$ is norm compact.
\end{prop}
\begin{proof}
    Let $\Phi\in C(\Omega, X)_{1}^*$ be a \textit{w}$^*$-\textit{w} PC. From Theorem~\ref{T23}, we have that $\Phi=\sum_{i=1}^{\infty}\alpha_{i}\delta_{k_i}\otimes x_{i}^* \in S(C(\Omega, X)^*)$, where each $\alpha_{i}\in (0,1]$ such that $\sum_{i=1}^{\infty}\alpha_{i}=1$, $k_{i}\in I$ and each $x_{i}^*\in S(X^*)$ is \textit{w}$^*$-\textit{w} PC. Suppose $\Phi$ attains its norm at an element $f\in S(C(\Omega, X))$. Then we have $$\Phi(f)=\sum_{i=1}^{\infty}\alpha_{i}x_{i}^*(f({k_{i}}))=1.$$ Since $\sum_{i=1}^{\infty}\alpha_{i}=1$, it gives that $x_{i}^*(f(k_i))=1$ for each $i\in \mathbb{N}$. Let $A=\{f(k_{i}):i\in \mathbb{N}\}$. Since $f$ is a continuous function and $\overline{A}\subseteq f(\overline{\{k_{i}:i\in\mathbb{N}\}})$, it gives that that $\overline{A}\subseteq X_{1}$ is a compact set.
\end{proof}
\begin{rem}
    In Proposition~\ref{T35}, if $X$ is a finite-dimensional space, then we do not need the additional assumption on $A$, that is, $\overline{A}$ is norm compact.
\end{rem}

In Proposition~\ref{T35}, we established a necessary condition under which a \textit{w}$^*$-\textit{w} PC $\Phi \in C(\Omega, X)^*$ attains its norm on $S(C(\Omega, X))$. However, for a general compact Hausdorff space $\Omega$, it is unclear whether this condition alone is sufficient for the converse. To address this question, consider the case where $S$ is a discrete set and $\beta S$ denotes the Stone-\v{C}ech Compactification of $S$. The following theorem establishes the converse of Proposition~\ref{T35} when $\Omega=\beta S$.

\begin{thm}\label{T0}
    Let $\Phi\in C(\beta S, X)_{1}^*$ be a \textit{w}$^*$-\textit{w} PC. If there exists a countable set $A=\{x_{i}:i\in \mathbb{N}\}_{}\subseteq X_{1}$ such that each element in range$(\Phi)$ attains its norm on $A$ and $\overline{A}\subseteq X_{1}$ is norm compact, then $\Phi$ attains its norm on $S(C(\beta S, X))$.
\end{thm}
\begin{proof}
Suppose $\Phi\in C(\beta S, X)_{1}^*$ is a \textit{w}$^*$-\textit{w} PC. By Theorem~\ref{T23}, we have that $\Phi = \sum_{i=1}^{\infty} \alpha_{i} \delta_{k_i} \otimes x_{i}^* \in S(C(\beta S, X)^*)$, where each $\alpha_{i} \in (0,1]$ with $\sum_{i=1}^{\infty} \alpha_{i} = 1$, $k_{i} \in S$, and each $x_{i}^* \in S(X^*)$ is a \textit{w}$^*$-\textit{w} PC. So the range$(\Phi)=\{x_{i}^*:i\in \mathbb{N}\}$. Let $A=\{x_{i}:i\in \mathbb{N}\}_{}\subseteq X_{1}$ be a countable set such that each element in range$(\Phi)$ attains its norm on $A$ and $\overline{A}\subseteq X_{1}$ is norm compact. For each $i\in\mathbb{N}$, there exists a $i_{j}\in\mathbb{N}$ such that $x_{i}^*$ attain its norm at $x_{i_{j}}$. So we have $x_{i}^*(x_{i_{j}})=1$ for each $i\in \mathbb{N}$. Define $f:S\rightarrow X_{1}$ by 
\[
    f(k)=
    \begin{cases}
        x_{i_{j}}, & \text{if } k=k_{i} \text{ for some } i\in \mathbb{N},\\
0, & \text{ otherwise}.
    \end{cases}
    \]
Since $S$ is a discrete set and range$(f)\subseteq \overline{A}\cup \{0\}$, we have that $f$ is a continuous function which takes values in a compact set. By the Dugundji's theorem (see \cite[Theorem~8.2]{Dug}), there exists a unique continuous function, say, $\tilde{f}: \beta S\rightarrow X_{1}$ such that $\tilde{f}|_{S}=f$. Now we have that 

\begin{align*}
\Phi(\tilde{f})&=\sum_{i=1}^{\infty}\alpha_{i}x_{i}^*(\tilde{f}({k_{i}}))\\
&=\sum_{i=1}^{\infty}\alpha_{i}x_{i}^*(f({k_{i}}))\\
&=\sum_{i=1}^{\infty}\alpha_{i}x_{i}^*(x_{i_{j}})\\
&=\sum_{i=1}^{\infty}\alpha_{i}\\
&=1.
\end{align*}
Thus, the measure $\Phi$ attains its norm at $\tilde{f}$. 
\end{proof}
Suppose $I\subset \Omega$ is the set of all isolated points, forming a discrete set. Using an argument similar to the one in the theorem above for $S=I$, we obtain a continuous function $\tilde{f}:\beta I\rightarrow X_{1}$. While the set $\beta I$ may not always be homeomorphic to a subset of $\Omega$, which is a key requirement to establish Theorem~\ref{T0}, the following proposition provides a situation in which such an extension is always possible.

\begin{prop}
    Let $\Omega$ be such that there exists a homeomorphism $\phi$ from $\beta I$ into $\Omega$ such that $\phi(k)=k$ for each $k\in I$. Let $\Phi\in C(\Omega, X)_{1}^*$ be a \textit{w}$^*$-\textit{w} PC. If there exists a countable set $A=\{x_{i}:i\in \mathbb{N}\}_{}\subseteq X_{1}$ such that each element in range$(\Phi)$ attains its norm on $A$ and $\overline{A}\subseteq X_{1}$ is norm compact, then $\Phi$ attains its norm.
\end{prop}
\begin{proof}
   Since $\Phi\in C(\Omega, X)_{1}^*$ is a \textit{w}$^*$-\textit{w} PC. Again, from Theorem~\ref{T23}, we have that $\Phi=\sum_{i=1}^{\infty}\alpha_{i}\delta_{k_i}\otimes x_{i}^* \in S(C(\Omega, X)^*)$, where each $\alpha_{i}\in (0,1]$ such that $\sum_{i=1}^{\infty}\alpha_{i}=1$, $k_{i}\in I$ and each $x_{i}^*\in S(X^*)$ is \textit{w}$^*$-\textit{w} PC. Let $A=\{x_{i}:i\in \mathbb{N}\}_{}\subseteq X_{1}$ satisfies the hypothesis of the theorem and let for each $i\in \mathbb{N}$, the point $x_{i_{j}}\in X_{1}$ be as in the proof of the above theorem. By using an argument similar to the one used in the theorem above for $S=I$, we obtain a continuous function $f:I\rightarrow X_{1}$ such that 
   \[
    f(k)=
    \begin{cases}
        x_{i_{j}}, & \text{if } k=k_{i} \text{ for some } i\in \mathbb{N},\\
0, & \text{ otherwise}.
    \end{cases}
    \]
    Suppose $\tilde{f}:\beta I\rightarrow X_{1}$ is an extension of $f$ as in the proof of the above theorem. By the hypothesis, let $\phi: \beta I\rightarrow \Omega$ be an into homeomorphism with $\phi(k)=k$ for each $k\in \mathbb{N}$, and hence, $\phi(k_{i})=k_{i}$ for each $i\in\mathbb{N}$. Now we consider the composition $\tilde{f}\circ \phi^{-1}:\phi(\beta I)\rightarrow X_{1}$. Since $\phi$ is a homeomorphism from $\beta I$ onto $\phi(\beta I)$, it gives that $\phi(\beta I)\subseteq \Omega$ is compact and the composition $\tilde{f}\circ \phi^{-1}$ is continuous. By using the vector-valued version of Tietze extension theorem (see \cite[Theorem~3.1]{FR}), we get a continuous map $g: \Omega \rightarrow X_{1}$ such that $\|g\|=\|\tilde{f}\|$ and $g|_{\beta I}=\tilde{f}|_{\beta I}$. Consequently, we have $$g(k_{i})=\tilde{f}(k_{i})=f(k_{i})=x_{i_{j}}$$ for each $i\in \mathbb{N}$. It is easy to see that the measure $\Phi$ attains its norm at $g$. This completes the proof.
\end{proof}

\section{More on w$^*$-w PCs and norm attainment}
In this section, we analyze the structure of \textit{w}$^*$-\textit{w} PCs in the dual unit ball $\textnormal{WC}(\Omega, X)_{1}^*$. As we have seen in the proof of Theorem~\ref{T23}, the functionals of the form $\delta_{\omega}\otimes x^*$ play a central role in characterizing the structure of \textit{w}$^*$-\textit{w} PCs. As shown in \cite[corollary~5]{DR}, we have the inclusions $C(\Omega, X)\subseteq \textnormal{WC}(\Omega, X)\subseteq C(\Omega, X)^{**}$, where $C(\Omega, X)\subseteq C(\Omega, X)^{**}$ is the canonical embedding. We now combine these observations to establish the structure of \textit{w}$^*$-\textit{w} PCs in the dual unit ball $\textnormal{WC}(\Omega, X)^*_{1}$.  
\begin{thm}\label{T41}
    Let $\Lambda\in\textnormal{WC}(\Omega,X)^*_{1}$ be a \textit{w}$^*$-\textit{w} PC. Then $\Lambda=\sum_{i=1}^{\infty}\alpha_{i}\delta_{\omega_{i}}\otimes x_{{i}}^*$, where $\alpha_{i}\in(0,1]$ such that $\sum_{i=1}^{\infty}\alpha_{i}=1$, $x_{{i}}^*\in X_{1}^*$ is \textit{w}$^*$-\textit{w} PC for each $i\in \mathbb{N}$ and each $\omega_{i}\in \Omega$ is an isolated point.
\end{thm}
\begin{proof}
    Let $\Lambda$ is \textit{w}$^*$-\textit{w} PC. In \cite{DR}, a projection is described as follows. $$P:\textnormal{WC}(\Omega,X)^*\to\textnormal{WC}(\Omega,X)^*$$ 

    For any $\mu\in \textnormal{WC}(\Omega,X)^*$, consider the measure $\nu=\mu|_{C(\Omega, X)}$. From \cite[Proposition~4]{DR}, we have that each $f\in\textnormal{WC}(\Omega,X)$ is Bochner $\nu$-integrable. Define
    $$
    P(\mu)(f)=\int{fd\nu}       \text{    for } f\in \textnormal{WC}(\Omega,X).
    $$

    Then $P$ is of norm one projection and Ker$(P)=C(\Omega, X)^{\perp}$. It is easy to see that range$(P)$ is isometric to $C(\Omega, X)^*$ and
$$
P(\delta_{\omega}\otimes x^*)(f)=x^*(f(\omega))=(\delta_{\omega}\otimes x^*)(f) \text{ for any } \omega\in\Omega, x^*\in X^*.
$$

If we denote $A=\{\delta_{\omega}\otimes x^*: \omega\in \Omega, x^*\in X^*\}$, then $\|f\|=\text{sup}\{x^*(f(\omega)):x^*\in X_{1}^* \text{ and } \omega\in \Omega\}$ for all $f\in \textnormal{WC}(\Omega,X)$. This implies that the set $A$ determines the norm on $ \textnormal{WC}(\Omega, X)$. By an application of the Hahn-Banach separation theorem, we have that $\overline{co}^{{\textit{w}}^*}(A)= \textnormal{WC}(\Omega, X)^*_{1}$. Since $\Lambda$ is a \textit{w}$^*$-\textit{w} PC, we have that $\Lambda\in \overline{co}^{w}(A)=\overline{co}^{\|.\|}(A)$. By using similar argument as given in the proof of Theorem~\ref{T23}, we get that $\Lambda=\sum_{i=1}^{\infty}\alpha_{i}\delta_{\omega_{i}}\otimes x_{{i}}^*$, where $\alpha_{i}\in(0,1]$ are such that $\sum_{i=1}^{\infty}\alpha_{i}=1$, $x_{{i}}^*\in X_{1}^*$ is \textit{w}$^*$-\textit{w} PC for each $i\in \mathbb{N}$ and each $\omega_{i}\in \Omega$ is an isolated point. This completes the proof.
\end{proof}

For a general compact Hausdorff space $\Omega$, the converse of Theorem~\ref{T41} may not be true. If $\Omega$ is an extremally disconnected space, one can use an argument similar to the one used in Theorem~\ref{T24} to obtain the converse. We provide these details for completeness.
\begin{thm}
    Let $\Omega$ be an extremally disconnected space and let $\Lambda \in \textnormal{WC}(\Omega, X)_{1}^*$ be such that $\Lambda=\sum_{i=1}^{\infty}\alpha_{i}\delta_{\omega_{i}}\otimes x_{{i}}^*$, where $\alpha_{i}\in(0,1]$ such that $\sum_{i=1}^{\infty}\alpha_{i}=1$, $x_{{i}}^*\in X_{1}^*$ is \textit{w}$^*$-\textit{w} PC for each $i\in \mathbb{N}$ and each $\omega_{i}\in \Omega$ is an isolated point. Then $\Lambda$ is a \textit{w}$^*$-\textit{w} PC.
\end{thm}
\begin{proof}
    We first assume that $I$ is dense in $\Omega$. Let $\Lambda=\sum_{i=1}^{\infty}\alpha_{i}\delta_{\omega_{i}}\otimes x_{{i}}^*$, where $\alpha_{i},\omega_{i} \text{ and } x_{i}^*$ satisfy the hypothesis of the theorem~\ref{T41}. As we have already shown in Theorem~\ref{T23} that $c_{0}(I,X)\subseteq C(\Omega, X)$ and it easy to see that $C(\Omega, X)\subseteq  \textnormal{WC}(\Omega,X)\subseteq l^{\infty}(I,X)$. Since $c_{0}(I,X)\subseteq l^{\infty}(I,X)$ is a M-ideal, we have $c_{0}(I,X)\subseteq  \textnormal{WC}(\Omega,X)$ is a M-ideal. Clearly, $\Lambda\in c_{0}(I,X)^*_{1}$. By using the similar argument as in the proof of Theorem~\ref{T23}, we have $\Lambda\in c_{0}(I,X)^*_{1}$ is \textit{w}$^*$-\textit{w} PC. Hence, the conclusion follows from Lemma~\ref{L21}.

    If $I$ is not dense in $\Omega$, since $\Omega$ is an extremally disconnected space, it follows that $\overline{I}$ is clopen. So we can write $\textnormal{WC}(\Omega, X)=WC(\overline{I},X)\oplus_{\infty}WC(\overline{I}^c,X)$. Now we may assume $\Lambda\in WC(\overline{I},X)_{1}^*$. By applying part (1) on the space $WC(\overline{I},X)$ gives that $\Lambda\in WC(\overline{I},X)_{1}^*$ is a \textit{w}$^*$-\textit{w} PC. Since an M-summand is also an M-ideal, the conclusion follows from Lemma~\ref{L21}.
\end{proof}

\begin{rem}
In Theorem~\ref{T0}, if we replace the norm compactness of $\overline{A}$ by weak compactness, then the same conclusion holds for the space $\textnormal{WC}(\Omega, X)$.    
\end{rem}
The following lemma is a well-known result that characterizes norm-attaining functionals in the space $c_{0}(X)$. We give a proof for the sake of completeness.
\begin{lem}\label{L31}
    Let $(x_{i}^*)\in S(c_{0}(X)^*)$ and $\Lambda=\{i\in \mathbb{N}:x_{i}^*=0\}$. Then $(x_{i}^*)$ attains its norm on $S(c_{0}(X))$ if and only if $\Lambda^c$ is a finite set.
\end{lem}
\begin{proof}
    The functional $(x_{i}^*)$ attains its norm if $(x_{i}^*)((x_{i}))=1$ for some $(x_{i})\in S(c_{0}(X))$. This implies,
    $$
         \sum_{i=1}^{\infty}x_{i}^*(x_{i})=1\leq\sum_{i=1}^{\infty}\|x_{i}^*\|\|x_{i}\|\leq\sum_{i=1}^{\infty}\|x_{i}^*\|\\.
    $$
    Since $\sum_{i=1}^{\infty}\|x_{i}^*\|=1$, we have $\|x_{i}^*\|\|x_{i}\|=\|x_{i}^*\|$ for each $i\in \mathbb{N}$. Suppose $\Lambda^c$ is an infinite set. Then we have $\|x_{i}\|=1$ for infinitely many $i\in \mathbb{N}$, which is a contradiction as $\|x_{i}\|\to 0$. Conversely, suppose $\Lambda^c$ is a finite set. For $i\in \Lambda^c$, we choose $x_{i}\in X_{1}$ such that $x_{i}^*(x_{i})=\|x_{i}^*\|$ and $x_{i}=0$ for each $i\in \Lambda$. It is easy to see that $(x_{i})\in S(c_{0}(X))$ and  $(x_{i}^*)$ attains its norm at $(x_{i})$.
\end{proof}
Recall that the space $c$ can be identified as $C(N_{+})$ where, $N_{+}$ is the one point compactification of $N$ (i.e, $N_{+}=N\cup\{\infty\})$. We use this identification to apply Theorem~\ref{T32} on the space $c$. In the next lemma, we extend the scalar-valued version of Lemma~\ref{L31} to the space $c$. 
\begin{prop}
    Let $(\alpha_{i})\in S(c^*)$ and $\Lambda=\{i\in \mathbb{N}:\alpha_{i}=0\}$. Then $(\alpha_{i})$ attains its norm on $S(c)$ if and only if either $\Lambda^c$ is finite or $\left(\frac{\overline{\alpha_{i_{k}}}}{|\alpha_{i_{k}}|}\right)_{k\in \mathbb{N}}$ is convergent, where $\alpha_{i_{k}}$ is the $k^{th}$ non zero entry of $(\alpha_{i})$.
\end{prop}
\begin{proof}
    Let $(\alpha_{i})$ attain its norm and suppose $\Lambda^c$ is an infinite set. By Theorem~\ref{T32} there exists $(\beta_{i})\in S(c)$ such that $|\beta_{i}|=1$ and $\sum_{i=1}^{\infty}\alpha_{i}\gamma_{i}=\sum_{i=1}^{\infty}\beta_{i}\gamma_{i}|\alpha_{i}|$ for each $(\gamma_{i})\in c$. Fix $k \in \mathbb{N}$, if we choose $(\gamma_{i})=e_{i_{k}}\overline{\beta_{i_{k}}}$, then we have $\alpha_{i_{k}}\overline{\beta_{i_{k}}}=|\alpha_{i_{k}}|$ or $\beta_{i_{k}}=\frac{{\alpha_{i_{k}}}}{|\alpha_{i_{k}}|}$. Since $(\beta_{i})$ is convergent, we have $\left(\frac{{\alpha_{i_{k}}}}{|\alpha_{i_{k}}|}\right)_{k\in \mathbb{N}}$ is convergent. Hence $\left(\frac{\overline{\alpha_{i_{k}}}}{|\alpha_{i_{k}}|}\right)_{k\in \mathbb{N}}$ is convergent.

    Conversely, if $\Lambda^c$ is finite, then by Lemma~\ref{L31}, $(\alpha_{i})$ attains its norm on $S(c_{0})$ hence in $S(c)$. Now suppose that $\Lambda^c$ is an infinite set and $\left(\frac{{\overline{\alpha_{i_{k}}}}}{|\alpha_{i_{k}}|}\right)_{k\in \mathbb{N}}$ is convergent to $\alpha_{0}$. Clearly, $|\alpha_{0}|=1$ define $z_{n}=\alpha_{0}$ if $n\notin \{i_{k}:k\in \mathbb{N}\}$ and $z_{i_{k}}=\frac{{\overline{\alpha_{i_{k}}}}}{|\alpha_{i_{k}}|}$. It is easy to see that $(z_{i})\in S(c)$ and $(\alpha_{i})$ attains its norm at $(z_{i})$. This completes the proof.
\end{proof}

\subsection{Open Problems: } In what follows, we present some open problems.
\begin{Q}
 For a general compact Hausdorff space $\Omega$, the converse of Theorem~\ref{T23} is not known.
\end{Q}

\begin{Q}
 For a general compact Hausdorff space $\Omega$, the converse of Theorem~\ref{T41} is not known.
\end{Q}

\begin{Q}
  In Theorem~\ref{T34}, we have characterize norm-attaining functionals in the space $C(\Omega, X)$, when $X^*$ has the RNP. Corresponding case where $X^*$ fails to have the RNP remain open.
\end{Q}

\begin{Q}
    For a general compact Hausdorff space $\Omega$, the converse of Theorem~\ref{T0} is not known.
\end{Q}

\small{\bf Data Availability Statement: }
No data was used for the research described in the article.

\small{\bf Conflict of interest:} The author declares that there is no conflict of interest.
\section*{Acknowledgements}
The author would like to thank Prof.~T.~S.~S.~R.~K.~Rao and Dr.~Priyanka Grover for their invaluable guidance, constant support, and inspiring mentorship throughout this work.  
The author is also grateful for their encouragement during the preparation of the thesis, of which this paper forms a part.  
Sincere thanks are due to the department for providing a stimulating academic environment that greatly facilitated this research.
\bibliographystyle{plain, abbrv}

\end{document}